\documentclass[11pt,letter]{amsart}

\usepackage[utf8]{inputenx}

\usepackage{appendix}
\usepackage{textcomp}
\usepackage[english]{babel}
\usepackage[T1]{fontenc}
\usepackage{bm}
\usepackage{amssymb,amsthm,amsmath}
\usepackage{indentfirst}
\usepackage{color}
\usepackage{graphicx}
\usepackage{enumerate}
\usepackage{enumitem}
\usepackage{caption}
\usepackage{empheq}
\usepackage{xcolor}
\usepackage{color}
\usepackage{amssymb}
\usepackage{booktabs}
\usepackage{graphicx}
\usepackage[font=small,labelfont=bf,labelsep=period,tableposition=top]{caption}
\usepackage{hyperref}

\definecolor{s1}{rgb}{0.91, 0.36, 0.47}
\definecolor{s2}{rgb}{0.35, 0.227, 0.33}
\definecolor{sf1}{rgb}{0.19, 0.14, 0.698}
\definecolor{sf2}{rgb}{0.61, 0.11, 0.08}

\newtheorem{theorem}{Theorem}[section]
\newtheorem{lemma}[theorem]{Lemma}

\newtheorem{remark}[theorem]{Remark}
\newtheorem{definition}{Definition}[section]
\newtheorem{corollary}[theorem]{Corollary}

\usepackage{ifthen}
\provideboolean{shownotes}\setboolean{shownotes}{true}
\newcommand{\margnote}[1]{\ifthenelse{\boolean{shownotes}}
	{\marginpar{\raggedright\tiny\texttt{#1}}}{}}

\usepackage{verbatim}

\newcommand{\CC}{\mathbb{C}}
\newcommand{\RR}{\mathbb{R}}

\newcommand{\bu}{\bm{u}}
\newcommand{\bn}{\bm{n}}
\newcommand{\bg}{\bm{g}}
\newcommand{\bx}{\bm{x}}

\newcommand{\pa}{\partial}

\definecolor{DarkGreen}{rgb}{0.04,0.7,0.07}
\definecolor{DarkOrange}{RGB}{255,140,0}

\usepackage{enumitem, hyperref}
\makeatletter
\def\namedlabel#1#2{\begingroup
	#2%
	\def\@currentlabel{#2}%
	\phantomsection\label{#1}\endgroup
}
\makeatother
\hypersetup
{   bookmarks=false,
	colorlinks=true,
	urlcolor=blue,
	citecolor=blue,
	pdfauthor = {Andrea Aspri},
	pdftitle = {dislocation draft},
	pdfcreator = {LaTeX, TeXmaker, pdfLaTeX, Hyperref}
}
\colorlet{linkequation}{blue}

\begin{document}

\title[Dislocations in an layered elastic medium]{Dislocations in a layered
elastic medium with applications to fault detection}

\author[A. Aspri {\em et al.}]{Andrea Aspri}
\address{Department of Mathematics, Universit\`a degli Studi di Pavia}
\email{andrea.aspri@unipv.it}

\author[]{Elena Beretta}
\address{Department of Mathematics, NYU-Abu Dhabi}
\email{elena.beretta@polimi.it}

\author[]{Anna L. Mazzucato}
\address{Department of Mathematics, Penn State University}
\email{alm24@psu.edu}

\date{\today}



\begin{abstract}
We consider a model for elastic dislocations in geophysics. We model a
portion of the Earth's crust as a bounded, inhomogeneous elastic
body with a buried fault surface, along which slip occurs. We prove
well-posedness of the resulting mixed-boundary-value-transmission problem, assuming only bounded elastic moduli. We establish uniqueness in the inverse problem of determining the fault surface and the slip
from a unique measurement of the displacement on an open patch at the surface,
assuming in addition that the Earth's crust is an isotropic, layered medium
with Lam\'e coefficients piecewise Lipschitz on a known
partition and that the fault surface satisfies certain geometric conditions.
These results substantially extend those of the authors in
{\em Arch. Ration. Mech. Anal.} {\bf 263} (2020), n. 1, 71--111.
\end{abstract}

\maketitle

\section{Introduction}
The focus of this work is an analysis of both the forward or direct problem, as
well as the inverse problem, for a model of buried faults in the Earth's crust.
Specifically, we prove well-posedness of the direct problem, assuming only
$L^\infty$ elastic coefficients, and uniqueness in the inverse problem, under
additional assumptions, which are motivated by the ill-posedness of the
inverse problem and are not overly restrictive for the applications we are
concerned about.

We model the Earth's crust as a layered, inhomogeneous
elastic medium, and the fault as an oriented, open  surface $S$ immersed in this
elastic medium and not reaching the surface (the case of {\em buried} or {\em blind faults}),
along which there can be slippage of the rock. Faults can have any
orientation with respect to the surface: horizontal, vertical, or oblique. When
slip occurs, we speak of elastic {\em dislocations}.
Mathematically, the slip is given by a non-trivial jump in the elastic
displacement across the fault, represented by a non-zero vector field $\bm{g}$
on $S$. The surface of the Earth can be assumed traction free, that is, no load
is bearing on it, while on the fault itself one can assume that the jump in the
traction is zero, that is, the loads on the two sides of the fault balances
out. (We refer the reader to \cite{Ciarlet,Gurtin} for instance for a
mathematical treatment of elasticity.)

The {\em direct} or {\em forward} problem consists in finding the elastic
displacement in the Earth's crust induced by the slip on the fault. The {\em
inverse} problem consists in determining the fault surface $S$ and the slip
$\bm{g}$ from measurements of surface displacement. The inverse problem has
important applications in seismology and geophysics.
The surface displacement can be inferred from Synthetic Aperture Radar (SAR)
and from Global Positioning System
(GPS) arrays monitoring (see e.g. \cite{SAR1,GPS2,GPS1,GPS3}).

In the so-called {\em interseismic} period, that is, the, usually long, period
between earthquakes, one can make a quasi-static approximation and work within
the framework of elastostatics. In seismology, the assumption of small
deformations is generally a good approximation away from active faults, and
therefore linear elastostatics is typically employed. Near active faults, and
especially during earthquakes, the so-called {\em co-seismic} period, more
accurate models assume the rock is viscoelastic. However, a rigorous analysis
of these more complex, non-linear models  is
still essentially missing. We plan to address non-linear and non-local models
in future work.

The study of elastic dislocations is classical in the context of isotropic,
homogeneous, linear elasticity, when
the surface $S$ is assumed to be of a particular simple form, that is, a
rectangular fault that has a not-too-big inclination angle with respect to the
unperturbed, flat Earth's surface.  (We refer to
\cite{Eshelby73, Segall10} and references therein for a more in-depth
discussion.)  In
this case, modeling the Earth's crust as an infinite half-space, there exists an
explicit formula for the displacement field induced by the slip on the fault,
due to Okada \cite{Okada92} (see also \cite{Nikkhoo17}).
To our knowledge, there are few works that tackle the forward problem in case of
non-homogeneous regular coefficients and more realistic geometries for the fault.
Indeed, the
problem is intrinsically singular along the fault, where non-standard
transmission conditions are imposed.
{A variational formulation of the problem for a bounded domain was introduced in  \cite{Zwieten_et_al14}.}

In \cite{Aspri-Beretta-Mazzucato-de-Hoop}, we proved well-posedness of the
direct problem for elastic dislocations, assuming the Earth's crust is an
infinite half-space, the elastic coefficients are Lipschitz continuous, and the
surface $S$ is also of Lipschitz class. We also established uniqueness in the
inverse problem from one measurement of surface displacement on an open patch,
under some additional assumptions on the geometry of the fault and the slip,
namely we took $S$ to be a graph with respect to an arbitrary, but given,
coordinate system, we assumed that $S$ has at least one corner singularity, and
assumed $\bm{g}$ tangential to $S$. The main difficulties in that work were
twofold. On one hand, we had to work with suitably weighted Sobolev spaces in
order to control the slow decay of solutions at infinity. On the other, we
allowed slips that do not vanish anywhere on $S$. Then the solution at the
boundary of the fault may develop singularities, for instance in the case of
constant slip and a rectangular fault, for which logarithmic blow-up at the
vertices exists, as noted already by Okada \cite{Okada92}. These potential singularities are
unphysical and  do not allow for a variational approach to well-posedness.
Instead, owing to the regularity of the coefficients, we used a duality
argument for an equivalent source problem. We also established a
double-layer-potential representation for the solution. Uniqueness for the
inverse problem was obtained using unique continuation, again owing to the
regularity of the coefficients.

The main focus of this work is to generalize the results in
\cite{Aspri-Beretta-Mazzucato-de-Hoop} to a more realistic set-up. We
model a portion of the Earth's crust, where the fault is located, as a Lipschitz
bounded domain $\Omega$, which includes the case of polyhedral domains,
relevant to numerical implementations and applications. The direct problem consists in solving a
mixed-boundary-value-transmission problem for the elasticity system in $\Omega$,
given in Equation \eqref{eq: Pu}. On the buried part of
the boundary of $\Omega$, which we call $\Sigma$, we impose homogeneous
Dirichlet boundary conditions, that is, zero displacement.
\begin{figure}
	\centering
	\includegraphics[scale=0.7]{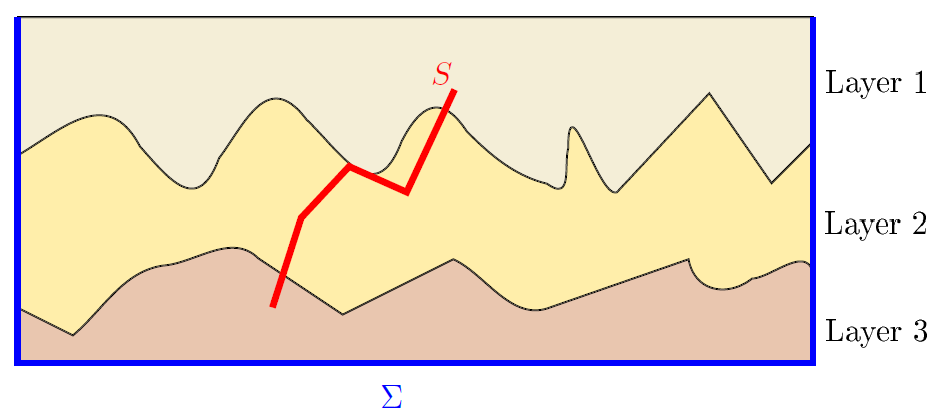}
	\caption{An example of the geometrical setting. A section of a layered medium with $S$, the dislocation surface, and with $\Sigma$, the buried part of $\Omega$.}\label{fig:geom_set}
\end{figure}
Such boundary condition model the situation, where the relative motion of rock formations is
small away from the fault as compared to that near the fault itself, except at
the surface of the Earth due to the traction-free assumption there. This assumption implicitly includes that $\Omega$ is large enough compared to the size of the active portion of the fault where the slippage occurred, so our model is not well suited for large, active faults.  A non-zero displacement on $\Sigma$ can also be imposed and other types of boundary conditions on $\Sigma$  can be treated, such as inhomogeneous Neumann boundary
conditions, modeling the load bearing on the rock formations at the boundary
from nearby formations. We assume the Earth's crust to be a layered elastic
medium, a common assumption in geophysics, that is, we assume that the elastic
coefficients are piecewise regular, but may jump across a known partition of
$\Omega$, see Figure \ref{fig:geom_set}, and impose standard transmission conditions at the interfaces of the
partition. This set-up has been considered in the literature to model
dislocations in geophysics (see for example
\cite{RMB02,Sato71,Zwieten_et_al13}). Furthermore, posing the problem in a bounded domain lends itself naturally to a numerical implementation of the problem that does not utilize boundary integral equations, rather it uses a variational formulation for the problem \cite{Antonietti-Bonaldi-Mazzieri,Quarteroni-et-al}.

In this work, we assume that the slip
$\bm{g}$ vanishes at the boundary of the fault. We are therefore modeling the case of an {\em unlocked fault patch} on only a part of the fault. By unlocked fault patch we mean a part of the fault surface where the rock of the two sides of the fracture have slipped freely relative one another. It is observed that most faults have a distribution of locked and unlocked patches. The quasi-static approximation can be used for so-called {\em aseismic creeping faults}, that is, faults where the rock slowly slips without major seismic events. There are known creeping faults in major populated areas of Japan and California, for instance. Creeping faults are relevant also for {\em microseismicity}, which indicates frequent seismic activity of small amplitude. (Among the vast literature on the subject, we refer the reader to  \cite{Faults1,Faults2,Faults3,Faults4}.)
The support of $\bm{g}$ can still
be the entire fault surface, a situation that arises in the inverse
problem. Then a variational solution exists for Problem \eqref{eq: Pu},
constructed by solving suitable auxiliary Neumann and mixed-boundary-value
problems, after \cite{Agranovich}. For the direct problem, well-posedness holds
if the elasticity tensor $\mathbb{C}$ is an (anisotropic) bounded, strongly
convex tensor.

For the inverse problem, we require more. On one hand, we need to guarantee
that unique continuation for the elasticity system holds. This can be achieved
by assuming that $\Omega$ is partitioned into finitely many Lipschitz
subdomains and assuming that $\mathbb{C}$ is isotropic with Lam\'e
coefficients Lipschitz continuous in each subdomain (see \cite{Alessandrini_et_al14,Beretta_et_al17,Beretta_et_al,Beretta_et_al14} where a similar approach has been used to determine internal properties of an elastic medium from boundary measurements). On the other,
the uniqueness proof, which uses an argument by contradiction, can be
guaranteed to hold when $S$ is a graph with respect to an arbitrary, but
chosen, coordinate system. This assumption is again not too restrictive in the
geophysical context and allows for an arbitrary orientation of the surface, horizontal, vertical, or oblique {(see Remark \ref{r:faultgeometry})}. Differently than in
\cite{Aspri-Beretta-Mazzucato-de-Hoop}, however, due to the fact that the slip
vanishes on the boundary of $S$, one does not need to assume $S$ has a corner singularity or assume a specific direction for the slip field $\bm{g}$. Therefore,
the results presented here are a substantial generalization over known results
for both the well-posedness of the direct and the uniqueness of the inverse
problem.

The inverse dislocation problem has been treated both within the mathematics
community \cite{Volkov-Voisin-Ionescu}, as well as in the
geophysics community (among the extensive literature we mention
\cite{Arnadottir-Segall, ONSKTI11, PTCSP19} and references therein).
Reconstruction has been tested primarily through iterative algorithms \cite{Volkov-Voisin-Ionescu}, based on
Newton's
methods or constrained optimization of a suitable misfit functional, using
either Boundary Integral methods or Finite Element methods, as well as
Green's function methods to solve the direct problem. For stochastic and
statistical approaches to inversion we mention \cite{MSB13, Volkov-Calafell}
and references therein.
We do not address here the question of reconstruction and its stability (see \cite{Beretta_et_al08,Triki-Volkov}). This
is focus of future work, which we plan to tackle by using appropriate iterative
algorithms and solving the direct problem via Discontinuous Galerkin methods
(for example adapting the methods in
\cite{Antonietti-Bonaldi-Mazzieri,Quarteroni-et-al}).

We close this Introduction with a brief outline of the paper. In Section
\ref{sec: notation and functional setting}, we introduce the relevant notation
and the function spaces used throughout. In Section \ref{sec:direct},
we discuss the main assumptions on the coefficients and the geometry, and we
address the well-posedness of the direct problem, while we discuss
additional assumptions and prove uniqueness for the inverse problem in Section
\ref{sec:inverse}.

\section*{Acknowledgments}
The authors thank E. Rosset and S. Salsa for suggesting relevant literature and
for useful discussions that lead us to improve some of the results in this
work. They also thank the anonimous referees for their insighful comments.
A. Aspri acknowledges the hospitality of the
Department of Mathematics at NYU-Abu Dhabi. A. Mazzucato was Visiting Professor at NYU-Abu Dhabi on leave from Penn State University, when part of this work was conducted. She is partially
supported by the US National Science Foundation Grant DMS-1615457 and
DMS-1909103.

\section{Notation and Functional Setting}\label{sec: notation and functional setting}
We begin by introducing needed notation and the functional setting for both the
direct as well as the inverse problem.

\textit{Notation:} We denote scalar quantities in italics, e.g. $\lambda, \mu, \nu$,
points and vectors in bold italics, e.g.  $\bm{x}, \bm{y}, \bm{z}$ and $\bm{u}, \bm{v}, \bm{w}$, matrices and
second-order tensors in bold face, e.g.  $\mathbf{A}, \mathbf{B}, \mathbf{C}$,
and fourth-order tensors in blackboard face, e.g.  $\mathbb{A}, \mathbb{B},
\mathbb{C}$.

The symmetric part of a second-order tensor $\mathbf{A}$ is denoted by
$\widehat{\mathbf{A}}=\tfrac{1}{2}\left(\mathbf{A}+\mathbf{A}^T\right)$, where
$\mathbf{A}^T$ is the transpose matrix. In particular, $\widehat{\nabla \bu}$
represents the deformation tensor.
We utilize  standard notation for inner products, that is,
$\bm{u}\cdot \bm{v}=\sum_{i} u_{i} v_{i}$, and  $\mathbf{A}:
\mathbf{B}=\sum_{i,j}a_{ij} b_{ij}$.
$|\mathbf{A}|$ denotes the norm induced by the inner product on matrices:
	\begin{equation*}
		|\mathbf{A}|=\sqrt{\mathbf{A}:\mathbf{A}}.
	\end{equation*}

\textit{Domains:} Given $r>0$, we denote  the ball of radius $r$ and
center
$\bm{x}$ by $	B_{r}(\bm{x})\subset\mathbb{R}^3$ and a circle of radius $r$
and center $\bm{y}$ by $B'_{r}(\bm{y})\subset\mathbb{R}^2$.
\begin{definition}[$C^{k,\alpha}$ regularity of domains]\ \\
		Let $\Omega$ be a bounded domain in $\mathbb{R}^3$. Given $k, \alpha$, with $k\in\mathbb{N}$ and $0<\alpha\leq 1$, we say that a portion $\Upsilon$ of $\partial \Omega$ is of class $C^{k,\alpha}$ with constant $r_0$, $E_0$, if for any $\bm{P}\in \Upsilon$, there exists a rigid transformation of coordinates under which we have that $\bm{P}$ is mapped to the origin  and
		\begin{equation*}
		\Omega\cap B_{r_0}(\bm{0})=\{\bm{x}\in B_{r_0}(\bm{0})\, :\, x_3>\psi(\bm{x}')\},
		\end{equation*}
		where $\bx=(x_1,x_2,x_3)$, $\bx'=(x_1,x_2,0)$ identified canonically with $(x_1,x_2)\in\RR^2$.
		Above, ${\psi}$ is a $C^{k,\alpha}$ function on $B'_{r_0}(\bm{0})\subset \mathbb{R}^2$, such that
		\begin{equation*}
		\begin{aligned}
		{\psi}(\bm{0})&=0,\\
		\nabla{\psi}(\bm{0})&=\mathbf{0}, \qquad \text{for}\, k\geq 1\\
		\|{\psi}\|_{C^{k,\alpha}(B'_{r_0}(\bm{0}))}&\leq E_0.
		\end{aligned}
		\end{equation*}
		When $k=0, \alpha=1$, we also say that $\Upsilon$ is of Lipschitz class with constants $r_0$, $E_0$.
	\end{definition}
Similarly, we define a surface $S$ of class $C^{k,\alpha}$ if it is locally the graph of a function $\psi$ with the properties described above. For $k=0$, $\alpha=1$, the case we are interested in, since the composition of Lipschitz maps is Lipschitz, we can define a Lipschitz curve on a Lipschitz surface $S$ (such as its boundary $\pa S$) by lifting a Lipschitz curve on $\RR^2$.

Given a bounded domain $\Omega\subset \mathbb{R}^3$ such that
$\overline{\Omega}:=\overline{\Omega^+}\cup\overline{\Omega^-}$, where
$\Omega^+$ and $\Omega^-$ are bounded domains, we call $f^+$ and $f^-$ the
restriction of a function or distribution $f$ to ${\Omega^+}$ and $\Omega^-$,
respectively.
We denote the jump of a function or tensor field
$\bm{f}$ across a bounded, oriented surface $S$ by $[\bm{f}]_S:=\bm{f}^+_S-\bm{f}^-_S$, where $\pm$ denotes a non-tangential limit to each side of the oriented surface
$S$, $S^+$ and $S^-$, where $S^+$ is by convention the side where the unit
normal vector  $\bn$ points into and $\bn$ is determined by the given
orientation on $S$.

\textit{Functional setting:} we use standard notation to denote the usual
functions spaces, e.g. $H^s(\Omega)$ denotes the $L^2$-based Sobolev space with
regularity index $s\in\mathbb{R}$. $C^\infty_0(\Omega)$ is the space of smooth functions
with compact support in $\Omega$.

We will need to consider trace spaces on open bounded
surfaces that have a good extension property to closed surfaces containing
them. (We refer to \cite{Lions-Magenes,Tartar} for an in-depth discussion).
In what follows, $D$ is a given open bounded Lipschitz domain in
$\mathbb{R}^n$, $n=2$ or $n=3$.

We recall that fractional Sobolev spaces on $D$ can be defined via real
interpolation, and that
$H^s_{0}(D):=\overline{C^{\infty}_0(D)}^{\|\cdot\|_{H^s(D)}}$,
$s\geq 0$. We also recall that $H^{s}(D)=H^{s}_0(D)$, $0\leq s\leq 1/2$. If $s<1/2$, it is possible to extend an
element of $H^s(D)$ by zero in $\RR^n\setminus D$ to an element of
$H^s(\mathbb{R}^n)$. When $s=1/2$ such an extension is possible for elements
that are suitably weighted by the distance to the boundary, since the extension operator from $H^{\frac{1}{2}}_0(D)$ to $H^{\frac{1}{2}}_0(\mathbb{R}^n)$ is not continuous. Following Lions and
Magenes \cite{Lions-Magenes},
we introduce the space $H^{\frac{1}{2}}_{00}(D)$ defined as
\begin{equation} \label{eq:H1200def}
  H^{\frac{1}{2}}_{00}(D):=\Big\{u\in H^{\frac{1}{2}}_{0}(D),
    \delta^{-1/2} u\in L^2(D)\Big\},
\end{equation}
where $\delta(x) =$dist$(x,\pa D)$, for $x\in D$.
This space is equipped with its natural norm, i.e.:
\[
    \|f\|_{H^{\frac{1}{2}}_{00}(D)} := \|f\|_{H^{\frac{1}{2}}(D)} +
           \|\delta^{-1/2}\,f\|_{L^2(D)},
\]
which gives a finer topology than that in $H^{\frac{1}{2}}(D)$. $\delta$ can be replaced by a function $\varrho\in C^\infty(D)$ that is comparable to the distance to the boundary, in the sense that
\begin{equation}
   \lim\limits_{x\to x_0} \frac{\varrho(x)}{d(x,\partial D)}=d\neq 0,\qquad
  \forall x_0\in \partial D,
\end{equation}
$\varrho>0$ in $D$ and $\varrho$ vanishes on $\pa D$ (see e.g. \cite[Lemma 3.6.1]{Ziemer}). We opted for the definition above of $H^{1/2}_{00}(D)$ as we do not need to consider higher-order traces.
If $v\in
H^{\frac{1}{2}}_{00}(D)$, then its extension by zero to $\RR^n\setminus D$
is an element of $H^{\frac{1}{2}}(\RR^n)$ and the extension operator is bounded. In
particular, $v=0$ on $\partial D$ in trace sense. The space $H^{\frac{1}{2}}_{00}$ can
also be identified with the real interpolation space (see e.g. \cite[Chapter 7]{Adams})
\begin{equation*}
   H^{\frac{1}{2}}_{00}(D)=(H^1_{0}(D), L^2(D))_{\frac{1}{2},2}.
\end{equation*}
We will need to define $ H^{\frac{1}{2}}_{00}(S)$ where $S$ is a Lipschitz surface. We can define
$ H^{\frac{1}{2}}_{00}(S)$ in a standard way using partitions of unity and coordinates charts  (see e.g. \cite{Adams,TriebelIII}). Then, we define $H^{\frac{1}{2}}_{00}(S)$ as in \eqref{eq:H1200def}, where $\delta$ is the distance induced by surface area on $S$. 
In addition, Lipschitz domains are extension domains for Sobolev spaces \cite{Ziemer}, so that using a local coordinate chart, we can prove the extension property from $H^{\frac{1}{2}}_{00}(S)$ to  $H^{\frac{1}{2}}(\Gamma)$, where $\Gamma$ is any closed surface containing $S$.

 Let $\mathcal{R}$ be the space of infinitesimal rigid motions in $\RR^3$. To
study the well-posedness of the direct problem, we introduce two
variational spaces
	\begin{equation}\label{def: sob_space_norm}
		\mathring{H}^1(D)=\Bigg\{\bm{\eta}\in H^1(D):\,\,\,
\int_{D}\bm{\eta}\cdot \bm{r}\, d\bm{x}=0, \forall \bm{r}\in\mathcal{R}\Bigg\},
	\end{equation}
and
	\begin{equation}\label{def: sob_space_trace}
	{H}^1_{\Sigma}(D)=\Bigg\{\bm{\eta}\in H^1(D):\,\,\,
\bm{\eta}|_{\Sigma}=\bm{0}\Bigg\},
	\end{equation}
where $\Sigma$ denotes the closure of an open subset of $\partial D$.

Finally, we denote the duality pairing between a Banach space $X$ and its dual
$X'$ with $\langle \cdot, \cdot \rangle_{(X',X)}$. When clear from the context,
we will omit the explicit dependence on the spaces, writing $\langle \cdot,
\cdot \rangle$. We will write $\langle \cdot, \cdot \rangle_D$ to denote the
pairing  restricted to a domain $D$.

\section{The direct problem} \label{sec:direct}
We first discuss the main assumptions on the dislocation surface $S$
and the elastic tensor $\mathbb{C}$, used in the rest of the paper. Then, we study the well-posedness of the forward problem.\\
Below $\Omega$ is a bounded Lipschitz domain.

\begin{description}
	\item[\namedlabel{hyp: assum_elastic_tens}{Assumption 1} - elasticity
	tensor]  The elasticity tensor $\mathbb{C}=\mathbb{C}(\bm{x})$ is a fourth-order tensor satisfying the full symmetry properties
	\begin{equation*}
	\mathbb{C}_{ijkh}(x)=\mathbb{C}_{jikh}(x)=\mathbb{C}_{khij}(x),\qquad \forall 1\leq i,j,k,h\leq 3,\ \textrm{and}\ \ \bm{x}\in\Omega,
	\end{equation*}
	which it is	assumed uniformly bounded, $\mathbb{C}\in L^{\infty}(\Omega)$, and
	uniformly strongly convex, that is,
	$\CC$ defines a positive-definite quadratic form on symmetric matrices:
	\begin{equation*}
	\mathbb{C}(\bm{x})\widehat{\mathbf{A}}:\widehat{\mathbf{A}}\geq c
	|\widehat{\mathbf{A}}|^{2}, \qquad \textrm{a.e in}\,\, \Omega,
	\end{equation*}
	for $c>0$.
\end{description}

\begin{description}
	\item[\namedlabel{hyp: assum_disl}{Assumption 2} - dislocation surface] We model the dislocation surface  $S$ by an open, bounded, oriented Lipschitz surface, with Lipschitz boundary, such that
	\begin{equation}
	\overline{S}\subset \Omega.
	\end{equation}
	We assume that  $S$ can be extended to a closed Lipschitz,
	orientable surface $\Gamma$ satisfying
	\begin{equation*}
	\Gamma \cap \partial \Omega=\emptyset.
	\end{equation*}
	Moreover, we indicate  the domain enclosed by $\Gamma$ with $\Omega^-$ and
	$\Omega^+=\Omega\setminus\overline{\Omega^-}$. We choose the orientation on $S$
	so that the associated normal $\bn$ coincides with the unit outer normal to
	$\Omega^{-}$.
\end{description}

In this section, we study the following mixed-boundary-value problem:
	\begin{equation}\label{eq: Pu}
		\begin{cases}
			\textrm{\textup{div}}\, (\mathbb{C}\widehat{\nabla}\bm{u})=\bm{0},& \textrm{in}\,\,
			\Omega \setminus \overline{S},\\
			(\mathbb{C}\widehat{\nabla}\bm{u})\bm{\nu}=\bm{0}, & \textrm{on}\,\, \partial \Omega \setminus \Sigma, \\
			\bm{u}=\bm{0}, & \textrm{on}\,\, \Sigma\\
			[\bm{u}]_{S}=\bm{g}, \\
			[(\mathbb{C}\widehat{\nabla}\bm{u})\bm{n}]_{S}=\bm{0} ,\\
		\end{cases}
	\end{equation}
where $\Sigma$ is the closure of an open subset in $\partial\Omega$, $\bm{n}$
is the normal vector induced by the orientation on $S$ (see \ref{hyp: assum_disl}), $\bm{\nu}$ is the unit
outer normal vector on  $\partial \Omega$.

$\Omega$ represents a portion of the Earth's crust where the fault $S$
lies and where both the direct and inverse problems are studied.
We assume that $S$ does not reach the boundary of $\Omega$, which corresponds geophysically to the case of buried or blind  faults. From the point of view of the inverse problem, this is the most interesting case, as there is no direct access to the fault from the boundary for monitoring.
$\Sigma$ models the buried part of the boundary of $\Omega$. Assuming that the rock
displacement is zero on $\Sigma$ is justified from a geophysical point of view, as
the relative motion of rock formations can be assumed much slower than rock
slippage along faults. In applications, one needs to assume there that $\Omega$ is large enough compared to the size of the fault for this justification to hold.
The complement of $\Sigma$ models the part of the
boundary on the Earth's crust and hence can be taken traction free.
(See e.g. \cite{Zwieten_et_al13}.)

The vector field $\bm{g}$ on $S$ models the slip along the active patch of the
fault. We assume that
\begin{equation}
\bm{g}\in H^{\frac{1}{2}}_{00}(S).
\end{equation}
Recall that elements in this space have zero trace at the boundary.

\begin{remark}\label{rem: extent_g}
By hypothesis (see \ref{hyp: assum_disl}), $S$ is part of a closed Lipschitz
surface $\Gamma$. Then, $\bm{g} \in H^{\frac{1}{2}}_{00}(S)$ implies that $\bm{g}$
can be extended by zero in $\Gamma\setminus S$ to a function $\widetilde{\bm{g}}
\in H^{\frac{1}{2}}(\Gamma)$:
		\begin{equation}\label{eq: g_extension}
			\widetilde{\bm{g}}(\bm{x})=
				\begin{cases}
				 \bm{g}(\bm{x}),& \textrm{if}\,\,\, \bm{x}\in S,\\
				 \bm{0},& \textrm{if}\,\,\, \bm{x}\in
\Gamma\setminus S.
				\end{cases}
		\end{equation}
\end{remark}

\begin{remark}
 As discussed in the Introduction, there are geophysical motivations for considering a slip $\bg$ that vanishes at the boundary of the surface $S$ (a creeping unlocked fault patch). There are mathematical reasons for considering such a class of slips, as well. Given the minimal regularity of the coefficients, a variational formulation of the problem is the most natural one.
However, it can be shown (see \cite{Aspri-Beretta-Mazzucato-de-Hoop} for a discussion on this point) that, if $\bg$ is an arbitrary field in $H^{\frac{1}{2}}(S)$, then the solution $\bu$ is not necessarily in $H^1(\Omega\setminus \overline{S})$. The space $H^{\frac{1}{2}}_{00}(S)$ is the optimal choice for the slip then, because it consists precisely of those elements in $H^{1/2}(S)$ that can be extended by zero to $H^{\frac{1}{2}}(\Gamma)$, where $\Gamma$ is any arbitrary Lipschitz closed surface containing $S$, with norm bounds on the extension and the restriction back to $\overline{S}$ (\cite{Tartar}). Moreover, we are also interested in implementing a reconstruction algorithm. If an iterative algorithm is used, then we need to numerically solve the forward or direct problem several times. For inhomogeneous media with discontinuous coefficients, as in this work, a variational approach, such as that in FEM and DG methods, is practical.  The advantage of working with $\Gamma$ instead of $S$ is that Green's formulas apply to $\Omega\setminus \Gamma$. By working with  $H^{\frac{1}{2}}_{00}(S)$, we are then able to prove the equivalence of the problem formulation using $\Gamma$ and that using $S$ at the level of weak solutions.
In the extensive literature concerning interface and boundary problems, other spaces have been considered, notably, the space $\{\bg \in H^{\frac{1}{2}}(\Gamma)\; \mid\; \text{supp}\, \bg \subseteq \overline{S}\}$ \cite{Agranovich} (we need to be able to take supp\,$\bg=\overline{S}$ for the inverse problem). It would be interesting and relevant to investigate further the relationship between this space and
$H^{\frac{1}{2}}_{00}(S)$. However, this analysis is not the focus of our work.
\end{remark}

By a {\it weak} solution of \eqref{eq: Pu} we mean that
 \begin{equation}
 \begin{cases}
 \textrm{\textup{div}}\, (\mathbb{C}\widehat{\nabla}\bm{u})=\bm{0},& \textrm{in}\,\,
 \left( H^1_{\Sigma}(\Omega)\right)'\\
 (\mathbb{C}\widehat{\nabla}\bm{u})\bm{\nu}=\bm{0}, & \textrm{in}\,\, H^{-\frac{1}{2}}(\partial \Omega \setminus \Sigma), \\
 \bm{u}=\bm{0}, & \textrm{in}\,\, H^{\frac{1}{2}}(\Sigma)\\
 [\bm{u}]_{S}=\bm{g},&  \textrm{in}\,\, H^{\frac{1}{2}}_{00}(S) \\
 [(\mathbb{C}\widehat{\nabla}\bm{u})\bm{n}]_{S}=\bm{0},& \textrm{in}\,\, H^{-\frac{1}{2}}(S) \\
 \end{cases}
 \end{equation}

The strategy that we follow here is an adaptation of the procedure described
in \cite{Agranovich_book} to solve classical transmission problems.
Given the closed surface $\Gamma$ and the extension \eqref{eq:
g_extension}, we decompose $\Omega$ into two domains $\Omega^-$ and $\Omega^+$,
as in \ref{hyp: assum_disl}. Then we construct a
weak solution of Problem \eqref{eq: Pu} by solving two boundary-value problems,
one in $\Omega^-$ and one $\Omega^+$, imposing suitable Neumann
conditions on $\Gamma$.
The key step in this procedure
consists in identifying the proper Neumann boundary condition on
$\Gamma$ such that
$[\bm{u}]_{\Gamma}=\bm{u}^+_{\Gamma}-\bm{u}^-_{\Gamma}=\widetilde{\bm{g}}$,
where $\bm{u}^+_{\Gamma}$ and $\bm{u}^-_{\Gamma}$ are the traces on
$\Gamma$ of the solutions $\bm{u}^+$ in $\Omega^+$
and $\bm{u}^-$ in $\Omega^-$, respectively.
In $\Omega^{-}$, the solution $\bm{u}^-$ will be sought in the
auxiliary space $\mathring{H}^1(\Omega^{-})$ to ensure uniqueness. This choice
imposes apparently artificial normalization conditions in $\Omega^{-}$, which
are not needed to solve the original problem \eqref{eq: Pu}. However, we
can verify {\em a posteriori} that such conditions are in fact satisfied by the
unique solution to the original problem.

We shall first  prove some preliminary results.

\begin{lemma}\label{lem: Htilde}
	Let $\overline{\Omega}=\overline{\Omega^+}\cup\overline{\Omega^-}$,
where $\Omega^+$ and $\Omega^-$ are defined in \ref{hyp: assum_disl}. Let
		\begin{equation}
			\widetilde{H}:=\Big\{f\in L^2(\Omega):\,\, f^+\in
H^1(\Omega^+),\, f^-\in H^1(\Omega^-),
\textup{\textrm{and}}\,\, [f]_{\Gamma\setminus\overline{S}}=0   \Big\},
		\end{equation}
where $f^+=f\lfloor_{\Omega^+}$ and $f^-=f\lfloor_{\Omega^-}$.
Then
		\begin{equation*}
			H^1(\Omega\setminus\overline{S})\cong \widetilde{H}.
		\end{equation*}
\end{lemma}

This result is classical (see e.g. \cite{Agranovich} for a proof using
Green's formula in Lipschitz domains, obtained in \cite{Necas}). We
include here the proof for the reader's sake.
\begin{proof}
Let $f\in \widetilde{H}$ and let $\bm{\varphi}\in
C^{\infty}(\overline{\Omega})$ with support in $\Omega\setminus
\overline{S}$. We apply the Divergence Theorem in $\Omega^+$ and $\Omega^-$,
obtaining
\begin{equation*}
    \int_{\Omega^-}\nabla f^-\cdot \bm{\varphi}\,
    d\bm{x}=\int_{\Gamma}f^-\bm{n}\cdot \bm{\varphi}^-\, d\sigma(\bm{x})-
    \int_{\Omega^-}f^-\, \textrm{div}\, \bm{\varphi}\, d\bm{x},
\end{equation*}
where $\bm{n}$ is the unit outer normal vector to $\Omega^-$.
Similarly
	\begin{equation*}
	\int_{\Omega^+}\nabla f^+\cdot \bm{\varphi}\, d\bm{x}=-\int_{\Gamma}f^+\bm{n}\cdot \bm{\varphi}^+\, d\sigma(\bm{x})- \int_{\Omega^+}f^+\, \textrm{div}\, \bm{\varphi}\, d\bm{x}.
	\end{equation*}
Therefore, we find
	\begin{equation*}
	\int_{\Omega^-}\nabla f^-\cdot \bm{\varphi}\, d\bm{x} +
\int_{\Omega^+}\nabla f^+\cdot \bm{\varphi}\, d\bm{x}=
\int_{\Gamma\setminus \overline{S}}\left(f^-\bm{\varphi}^--
f^+\bm{\varphi}^+\right)\cdot \bm{n}\, d\sigma(\bm{x})-\int_{\Omega\setminus
\overline{S}}f\, \textrm{div}\,\bm{\varphi}\, d\bm{x},
	\end{equation*}
noting in the terms on the right that $\bm{\varphi}$ and $\textrm{div}\,
\bm{\varphi}$ have compact support in $\Omega\setminus\overline{S}$ and, as an
$L^2$ function, $f=f^+\chi_{\Omega^+}+f^-\chi_{\Omega^-}$. Moreover,
$\bm{\varphi}$
is regular across $\Gamma\setminus\overline{S}$ by hypothesis and $f^+=f^-$ on
$\Gamma\setminus\overline{S}$, since $f\in\widetilde{H}$, so that
	\begin{equation*}
		\int_{\Gamma\setminus \overline{S}}\left(f^-\bm{\varphi}^--
f^+\bm{\varphi}^+\right)\cdot \bm{n}\,
d\sigma(\bm{x})=\int_{\Gamma\setminus\overline{S}}
\bm{\varphi}\,\left(f^- - f^+\right)\cdot \bm{n}\,
d\sigma(\bm{x})=0.
	\end{equation*}
Consequently,
	\begin{equation*}
		\int_{\Omega\setminus\overline{S}}f\, \textrm{div}\, \bm{\varphi}\, d\bm{x}=-\int_{\Omega^+}\nabla f^+ \cdot \bm{\varphi}\, d\bm{x}-\int_{\Omega^-}\nabla f^- \cdot \bm{\varphi}\, d\bm{x},
	\end{equation*}
which means that the distributional gradient of $f$ is an $L^2$ function in
$\Omega\setminus \overline{S}$ and agrees with
	\begin{equation*}
		\nabla f^+ \chi_{\Omega^+}+\nabla f^- \chi_{\Omega^-}.
	\end{equation*}
Reversing the argument gives the opposite implication.
\end{proof}

For the next lemma, we follow \cite[Proposition 12.8.2]{Agranovich_book},
adapting that result to the case of the Lam\'e operator with discontinuous
coefficients.

\begin{lemma}\label{lem: jump_conor}
	Let $\mathbb{C}\in L^{\infty}(\Omega)$, and let $\bm{\eta}\in
{H}^1(\Omega\setminus\overline{S})$ be a weak solution of the system
$\textrm{div}(\mathbb{C}\widehat{\nabla}\bm{\eta})=\bm{0}$ in
$\Omega\setminus\overline{S}$.
	Then
$[(\mathbb{C}\widehat{\nabla}\bm{\eta})\bm{n}]_{\Gamma\setminus\overline{S}}=\bm
{0}$ in $H^{-\frac{1}{2}}(\Gamma\setminus\overline{S})$.
\end{lemma}

\begin{proof}
	We fix a point $\bm{x}_0\in \Gamma\setminus\overline{S}$ and we
consider a ball $B_r(\bm{x}_0)$ with $r>0$ sufficiently small so that
$B_r(\bm{x}_0)\cap \overline{S}=\emptyset$ and
$B_r(\bm{x}_0)\cap\partial\Omega=\emptyset$.\\
Let $\bm{\varphi}\in H^1_0(B_r(\bm{x}_0))$, then
	\begin{equation}\label{eq: weak_form1}
	\begin{aligned}
	0=\langle\textrm{div}(\mathbb{C}\widehat{\nabla}\bm{\eta}),\bm{\varphi}\rangle=-\int_{B_r(\bm{x}_0)}\mathbb{C}\widehat{\nabla}\bm{\eta}:\widehat{\nabla}\bm{\varphi}\, d\bm{x}.
	\end{aligned}
	\end{equation}
(This identity can be established by approximating $\bm{\varphi}$ with smooth
fields supported in $B_r(\bm{x}_0)$.)
Next we apply Green's identities, which hold for $H^1$-functions, in
$D^+=B_r(\bm{x}_0)\cap \Omega^+$ and $D^-=B_r(\bm{x}_0)\cap \Omega^-$.
Therefore, for all $\bm{\varphi}\in H^1_0(B_r(\bm{x}_0))$,
\begin{equation}\label{eq: weak_form2}
	0=-\int_{D^+}\mathbb{C}\widehat{\nabla}\bm{\eta}:\widehat{\nabla}\bm{\varphi}\, d\bm{x}- \langle \mathbb{C}\widehat{\nabla}\bm{\eta}^+\bm{n},\bm{\varphi}^+\rangle_{(H^{-\frac{1}{2}}(\Gamma\cap \partial D^+),H^{\frac{1}{2}}(\Gamma\cap \partial D^+))}
\end{equation}
and, analogously,
\begin{equation}\label{eq: weak_form3}
	0=-\int_{D^-}\mathbb{C}\widehat{\nabla}\bm{\eta}:\widehat{\nabla}\bm{\varphi}\, d\bm{x}+ \langle \mathbb{C}\widehat{\nabla}\bm{\eta}^-\bm{n},\bm{\varphi}^-\rangle_{(H^{-\frac{1}{2}}(\Gamma\cap \partial D^-),H^{\frac{1}{2}}(\Gamma\cap \partial D^-))}.
\end{equation}
Since $\bm{\varphi}\in H^1_0(B_r(\bm{x}_0))$,
$\bm{\varphi}^+_{|_{\Gamma\cap B_r(\bm{x}_0)}}=\bm{\varphi}^-_{|_{\Gamma\cap
B_r(\bm{x}_0)}}=:\bm{\varphi}_{|_{\Gamma\cap B_r(\bm{x}_0)}}$. Hence, adding
\eqref{eq: weak_form2} and \eqref{eq: weak_form3} gives:
\begin{equation*}
0=-\int_{D^+}\mathbb{C}\widehat{\nabla}\bm{\eta}:\widehat{\nabla}\bm{\varphi}\,
d\bm{x}-\int_{D^-}\mathbb{C}\widehat{\nabla}\bm{\eta}:\widehat{\nabla}\bm{
\varphi}\, d\bm{x}- \langle
[\mathbb{C}\widehat{\nabla}\bm{\eta}\bm{n}],\bm{\varphi}\rangle_{(H^ { -\frac{1}{2} }
(\Gamma\cap B_r(\bm{x}_0)),H^{\frac{1}{2}}(\Gamma\cap B_r(\bm{x}_0)))},
\end{equation*}
where $[,]$ denotes the jump across $\Gamma\cap
B_r(\bm{x}_0)$.
Consequently,
\begin{equation*}
0=-\int_{B_r(\bm{x}_0)}\mathbb{C}\widehat{\nabla}\bm{\eta}:\widehat{\nabla}\bm{
\varphi}\, d\bm{x} - \langle
[\mathbb{C}\widehat{\nabla}\bm{\eta}\bm{n}],\bm{\varphi}\rangle_{(H^{-\frac{1}{2}}
(\Gamma\cap B_r(\bm{x}_0)),H^{\frac{1}{2}}(\Gamma\cap B_r(\bm{x}_0)))},
\end{equation*}
using that, by hypothesis, both $\nabla \bm{\varphi}$ and $\nabla \bm{\eta}$
exists as $L^2$ functions in $B_r(\bm{x}_0)$.
From \eqref{eq: weak_form1} it follows that
\begin{equation*}
      \langle
      [\mathbb{C}\widehat{\nabla}\bm{\eta}\bm{n}],\bm{\varphi}\rangle_{(H^{-\frac{1}{2}}
     (\Gamma\cap B_r(\bm{x}_0)),H^{\frac{1}{2}}(\Gamma\cap B_r(\bm{x}_0)))}=0.
\end{equation*}
Since $\bm{\varphi}$ is an arbitrary function in $H^1_0(B_r(\bm{x}_0))$, we
have that $[\mathbb{C}\widehat{\nabla}\bm{\eta}\bm{n}]_{{\Gamma\cap
B_r(\bm{x}_0)}}=\bm{0}$ in $H^{-\frac{1}{2}}(\Gamma\setminus\overline{S})$. We conclude by covering $\Gamma\setminus
\overline{S}$ with a finite number of balls $B_{r}(\bm{x}_i)$, $i=1,\ldots, N$.
\end{proof}

We are now ready to tackle the well-posedness of Problem \eqref{eq: Pu}.
We begin by addressing the uniqueness of weak solutions.

\begin{theorem}(Uniqueness)\label{th: uniqueness}
	Problem \eqref{eq: Pu} has at most one weak solution in
\mbox{${H}^1_{\Sigma}(\Omega\setminus\overline{S})$}.
\end{theorem}
\begin{proof}
Assume that there exist two solutions $\bm{u}^1,\bm{u}^2\in
{H}^1_{\Sigma}(\Omega\setminus\overline{S})$. Let
$\bm{v}=\bm{u}^1-\bm{u}^2$. From the transmission conditions on $S$ (see \eqref{eq: Pu}), we have
\begin{equation*}
[\bm{v}]_S=\bm{0},\qquad\qquad [(\mathbb{C} \widehat{\nabla} \bm{v}) \bm{n}]_{S}= \bm{0}.
\end{equation*}
Hence, by Lemma \ref{lem: Htilde},
$\bm{v}\in{H}^1_{\Sigma}(\Omega)$. It follows that
$\bm{v}$ is a weak solution of the
problem
\begin{equation}\label{eq: eq_v}
\begin{cases}
\textrm{div}(\mathbb{C}\widehat{\nabla}\bm{v})=\bm{0}\, & \text{in}\,\,
\Omega\\
(\mathbb{C}\widehat{\nabla}\bm{v})\bm{\nu}=\bm{0} & \textrm{on}\,\,
\partial\Omega\setminus\Sigma\\
\bm{v}=\bm{0} & \textrm{on}\,\, \Sigma,
\end{cases}
\end{equation}
which has a unique solution, $\bm{v}=\bm{0}$.

\end{proof}

\begin{theorem}(Existence)\label{th: existence}
	There exists a weak solution $\bm{u}\in {H}^1_{\Sigma}(\Omega\setminus
\overline{S})$ to Problem  \eqref{eq: Pu}.
\end{theorem}

\begin{proof}
The strategy is to construct a weak solution of Problem \eqref{eq: Pu} from the solutions of two auxiliary two boundary-value problems, one in $\Omega^-$ and one $\Omega^+$, that are connected through a suitably chosen Neumann boundary condition on the interface $\Gamma$.
Specifically, we consider the following Neumann boundary-value problem in
$\mathring{H}^1(\Omega^-)$ (the space is chosen in order to avoid rigid motions
in $\Omega^-$):
	\begin{equation}\label{eq: P_Neumann1}
		\begin{cases}
		\textrm{div}(\mathbb{C}\widehat{\nabla}\bm{u}^-)=\bm{0}, & \textrm{in}\,\,\, \Omega^-\\
		(\mathbb{C}\widehat{\nabla}\bm{u}^-)\bm{n}=\bm{\varphi}, & \textrm{on}\,\,\, \Gamma
		\end{cases}
	\end{equation}
and the following mixed-boundary-value problem in  ${H}^1_{\Sigma}(\Omega^+)$:
	\begin{equation}\label{eq: P_Neumann2}
	\begin{cases}
	\textrm{div}(\mathbb{C}\widehat{\nabla}\bm{u}^+)=\bm{0}, & \textrm{in}\,\,\, \Omega^+\\
	(\mathbb{C}\widehat{\nabla}\bm{u}^+)\bm{\nu}=\bm{0}, & \textrm{on}\,\,\, \partial\Omega\setminus\Sigma\\
	\bm{u}^+=\bm{0}, & \textrm{on}\,\,\, \Sigma\\
	(\mathbb{C}\widehat{\nabla}\bm{u}^+)\bm{n}=\bm{\varphi}, & \textrm{on}\,\,\, \Gamma,
	\end{cases}
	\end{equation}
where $\Omega^-$ and $\Omega^+$ are defined in \ref{hyp:
assum_disl}. We denote the traces of $\bu^\pm$ in $H^{\frac{1}{2}}(\Gamma)$ by $\bu^\pm_\Gamma$.
The key point of the proof is to identify $\bm{\varphi}$ in order to represent
the solution $\bm{u}$ of \eqref{eq: Pu} as
$\bm{u}=\bm{u}^-\chi_{\Omega^-}+\bm{u}^+\chi_{\Omega^+}$, where
$\chi_{\Omega^-}$ and $\chi_{\Omega^+}$ are the characteristic functions of the
sets $\Omega^-$ and $\Omega^+$, respectively.
To this end, we define the bounded Neumann-Dirichlet operators
	\begin{equation*}
		N^+:H^{-\frac{1}{2}}(\Gamma)\to H^{\frac{1}{2}}(\Gamma),\qquad\qquad
		N^-:H^{-\frac{1}{2}}(\Gamma)\to H^{\frac{1}{2}}(\Gamma),
	\end{equation*}
related to \eqref{eq: P_Neumann2} and \eqref{eq: P_Neumann1}, respectively. Then, since $N^+\bm{\varphi}=\bm{u}^+_{\Gamma}$ and $N^-\bm{\varphi}:=\bm{u}^-_{\Gamma}$, and recalling that $[\bm{u}]_\Gamma=\widetilde{\bm{g}}$, where $\widetilde{\bm{g}}$ is the extension of $\bm{g}$ on $\Gamma\setminus S$, as defined in \eqref{eq: g_extension},
we need to identify $\bm{\varphi}\in H^{-\frac{1}{2}}(\Gamma)$ such that
			\begin{equation}\label{eq: oper_eq}
				\bm{u}^+_{\Gamma}-\bm{u}^-_{\Gamma}=(N^+-N^-)\bm{\varphi}=\widetilde{\bm{g}}.
			\end{equation}
\\
The invertibility of the operator $N^+-N^-$ guarantees that
$\bm{\varphi}=(N^+-N^-)^{-1}(\widetilde{\bm{g}})$, and  follows from the
continuity of both the Neumann-to-Dirichlet and the Dirichlet-to-Neumann
maps. Such continuity is well known. We briefly outline here a proof of
invertibility  in our setting for the reader's sake.

First, by using the weak formulation of
\eqref{eq: P_Neumann1} in $\mathring{H}^1(\Omega^-)$  and \eqref{eq: P_Neumann2}
in ${H}^1_{\Sigma}(\Omega^+)$, we find a relation between the quadratic form
associated to \eqref{eq: P_Neumann1} and  $\langle \bm{\varphi},
N^-\bm{\varphi}\rangle_{\Gamma}$, and between the quadratic form
associated to \eqref{eq: P_Neumann2}  and $\langle
\bm{\varphi}, N^+\bm{\varphi}\rangle_{\Gamma}$.
Indeed, from the weak formulation of problems \eqref{eq: P_Neumann1} and
\eqref{eq: P_Neumann2}, we find that
	\begin{equation}\label{eq: weak_form_+}
		\int_{\Omega^+}\mathbb{C}\widehat{\nabla}\bm{u}^+\cdot
\widehat{\nabla}\bm{v}^+\, d\bm{x}=-\langle \bm{\varphi},
\bm{v}^+\rangle_{\Gamma},\qquad\qquad \forall \bm{v}^+\in
{H}^1_{\Sigma}(\Omega^+),
	\end{equation}
as $\bm{n}$ points inwards in $\Omega^{+}$,
and that
	\begin{equation}\label{eq: weak_form_-}
	\int_{\Omega^-}\mathbb{C}\widehat{\nabla}\bm{u}^-\cdot
\widehat{\nabla}\bm{v}^-\, d\bm{x}=\langle \bm{\varphi},
\bm{v}^-\rangle_{\Gamma},\qquad\qquad \forall \bm{v}^-\in
 H^1(\Omega^-).
	\end{equation}
Next, we observe that we can extend any function $\bm{v}\in H^{\frac{1}{2}}(\Gamma)$ to,
respectively, functions $\bm{v}^+\in H^1_{\Sigma}(\Omega^+)$ and
$\bm{v}^-\in H^1(\Omega^-)$, for instance by solving suitable Dirichlet
problems for the Laplace operator in $\Omega^+$ and $\Omega^-$. Then, the
above identities imply:
\[
     |\langle \bm{\varphi}, \bm{v}\rangle_{\Gamma}|\leq
  C_{\pm}\,\|\bm{u}^\pm\|_{H^1(\Omega^\pm)} \,
\|\bm{v}^\pm\|_{H^1(\Omega^{\pm})}
   \leq C_{\pm}\, \|\bm{u}^\pm\|_{H^1(\Omega^\pm)} \,
\|\bm{v}\|_{H^{\frac{1}{2}}(\Gamma)}.
\]
Using the definition of the norm in
$H^{-\frac{1}{2}}(\Gamma)$ as the operator norm of functionals on $H^{\frac{1}{2}}(\Gamma)$,
it follows that
\begin{equation} \label{eq:Phi_norm}
    \|\bm{\varphi}\|_{H^{-\frac{1}{2}}(\Gamma)} \leq C_{\pm}
 \|\bm{u}^{\pm}\|_{H^1(\Omega^{\pm})}.
\end{equation}
Moreover, by choosing $\bm{v}^+=\bm{u}^+$ in \eqref{eq:
weak_form_+} and $\bm{v}^-=\bm{u}^-$  in
\eqref{eq: weak_form_-}, we have:
\begin{equation*}
\int_{\Omega^+}\mathbb{C}\widehat{\nabla}\bm{u}^+\cdot
\widehat{\nabla}\bm{u}^+\, d\bm{x}=-\langle \bm{\varphi},
N^+\bm{\varphi}\rangle_{\Gamma},
\end{equation*}
and
\begin{equation*}
\int_{\Omega^-}\mathbb{C}\widehat{\nabla}\bm{u}^-\cdot
\widehat{\nabla}\bm{u}^-\, d\bm{x}=\langle \bm{\varphi},
N^-\bm{\varphi}\rangle_{\Gamma}.
\end{equation*}
Then \ref{hyp: assum_elastic_tens}, Korn's and Poincar\'e's inequalities
(see e.g.
\cite{OleinikShamaevYosifian})
give that
\begin{equation}\label{eq:coerc_N+}
		-\langle \bm{\varphi},
N^+\bm{\varphi}\rangle_{\Gamma}=\int_{\Omega^+}\mathbb{C}\widehat{\nabla}\bm{u}
^+\cdot \widehat{\nabla}\bm{u}^+\, d\bm{x}\geq C
\|\bm{u}^+\|^2_{{H}^1(\Omega^+)},
\end{equation}
and that
\begin{equation}\label{eq:coerc_N-}
	\langle \bm{\varphi},
N^-\bm{\varphi}\rangle_{\Gamma}=\int_{\Omega^-}\mathbb{C}\widehat{\nabla}\bm{u}
^-\cdot \widehat{\nabla}\bm{u}^-\, d\bm{x}\geq C
\|\bm{u}^-\|^2_{\mathring{H}^1(\Omega^-)}.
\end{equation}
Therefore, by using \eqref{eq:Phi_norm} in both
\eqref{eq:coerc_N+} and  \eqref{eq:coerc_N-}, we can
establish the coercivity of the bilinear form associated to Equation
\eqref{eq: oper_eq}, that is,
\begin{equation}
	\|\bm{\varphi}\|^2_{H^{-\frac{1}{2}}(\Gamma)}\leq C \langle \bm{\varphi},
     (-N^++N^-)\bm{\varphi}\rangle_{\Gamma}.
\end{equation}
The continuity of this form follows directly from the continuity of the
solution operators for \eqref{eq: P_Neumann1}-\eqref{eq: P_Neumann2} and the
Trace Theorem.
The Lax-Milgram Theorem then ensures that there exists a unique solution
$\bm{\varphi}\in H^{-\frac{1}{2}}(\Gamma)$ such that
\begin{equation}
		\langle \bm{\psi}, (-N^++N^-)\bm{\varphi}\rangle_{\Gamma}=\langle \bm{\psi},-\widetilde{\bm{g}}\rangle_{\Gamma},\qquad \forall \bm{\psi}\in H^{-\frac{1}{2}}(\Gamma),
\end{equation}
namely the operator $-N^++N^-$ is invertible.

With this choice of $\bm{\varphi}$, Problems \eqref{eq:
P_Neumann1} and \eqref{eq: P_Neumann2} admit  unique solutions $\bm{u}^-\in
\mathring{H}^1(\Omega^-)$ and $\bm{u}^+\in {H}^1_{\Sigma}(\Omega^+)$,
respectively. Next, we let
\begin{equation*}
     \bm{u}=\bm{u}^-\chi_{\Omega^-}+\bm{u}^+\chi_{\Omega^+}.
\end{equation*}
Then, $\bm{u}=\bm{u}^-\in H^1(\Omega^-)$, $\bm{u}=\bm{u}^+\in
H^1(\Omega^+)$,  $\bm{u}$ is a distributional solution of
$\text{div}(\mathbb{C} \widehat \nabla \bm{u})= \bm {0}$ in $\Omega^+$ and
$\Omega^-$. To conclude, we show that $\bm{u}$ is a weak solution of \eqref{eq: Pu}. By construction, it follows that $\bm{u}$ satisfies the boundary conditions on $\partial\Omega$ in trace
sense. Again by construction
\begin{equation*}
[\bm{u}]_{\Gamma}=\bm{u}^+_{\Gamma}-\bm{u}^-_{\Gamma}=\widetilde{\bm{g}}
\quad \text{in } H^{\frac{1}{2}}(\Gamma).
\end{equation*}
That is, by \eqref{eq: g_extension},
\begin{equation}\label{eq: null_jump}
[\bm{u}]_{\Gamma\setminus \overline{S}}=\bm{0},\qquad\qquad [\bm{u}]_{{S}}=\bm{g},
\end{equation}
hence, by Lemma \ref{lem: Htilde}, it follows
that $\bm{u}\in H^1(\Omega\setminus\overline{S})$.
Moreover,
\begin{equation}\label{eq: jump_conorm}
   [\mathbb{C}\widehat{\nabla}\bm{u}\,\bm{n}]_{\Gamma}=\bm{0} \quad \text{in }
    H^{-\frac{1}{2}}(\Gamma),
\end{equation}
which follows immediately by construction. In particular,
$[\mathbb{C}\widehat{\nabla}\bm{u}\,\bm{n}]_{S}=\bm{0}$.
Now, recalling that $\bm{u}$ is a weak solution in $\Omega^-$ and in $\Omega^+$
and satisfies \eqref{eq: jump_conorm}, reversing the steps in the proof of
Lemma \ref{lem: jump_conor}, we obtain that $\bm{u}$ is a weak solution of
$\textrm{div}(\mathbb{C}\widehat{\nabla}\bm{u})=\bm{0}$ in
$\Omega\setminus\overline{S}$. In fact, we fix a point $\bm{x}_0\in\Gamma\setminus\overline S$ and we consider a ball $B_r(\bm{x}_0)$ with $r>0$ sufficiently small such that $B_r(\bm{x}_0)\cap \overline{S}=\emptyset$. Let $\varphi\in H^1_0(B_r(\bm{x}_0))$. We apply Green's identity in $D^+=B_r(\bm{x}_0)\cap \Omega^+$ and $D^-=B_r(\bm{x}_0)\cap \Omega^-$. Since $\bm{u}$ is a weak solution in $D^+$ and $D^-$, we get
\begin{equation}\label{eq: d+}
0=-\int_{D^+}\mathbb{C}\widehat{\nabla}\bm{u}:\widehat{\nabla}\bm{\varphi}\, d\bm{x}- \langle \mathbb{C}\widehat{\nabla}\bm{u}^+\bm{n},\bm{\varphi}^+\rangle_{(H^{-\frac{1}{2}}(\Gamma\cap \partial D^+),H^{\frac{1}{2}}(\Gamma\cap \partial D^+))}
\end{equation}
for all $\bm{\varphi}\in H^1_0(B_r(\bm{x}_0))$ and, analogously,
\begin{equation}\label{eq: d-}
0=-\int_{D^-}\mathbb{C}\widehat{\nabla}\bm{u}:\widehat{\nabla}\bm{\varphi}\, d\bm{x}+ \langle \mathbb{C}\widehat{\nabla}\bm{u}^-\bm{n},\bm{\varphi}^-\rangle_{(H^{-\frac{1}{2}}(\Gamma\cap \partial D^-),H^{\frac{1}{2}}(\Gamma\cap \partial D^-))}.
\end{equation}
As $\bm{\varphi}\in H^1_0(B_r(\bm{x}_0))$,
$\bm{\varphi}^+_{|_{\Gamma\cap B_r(\bm{x}_0)}}=\bm{\varphi}^-_{|_{\Gamma\cap
		B_r(\bm{x}_0)}}=:\bm{\varphi}_{|_{\Gamma\cap B_r(\bm{x}_0)}}$ in the trace sense. Hence, adding
\eqref{eq: d+} and \eqref{eq: d-} gives:
\begin{equation*}
0=-\int_{D^+}\mathbb{C}\widehat{\nabla}\bm{u}:\widehat{\nabla}\bm{\varphi}\,
d\bm{x}-\int_{D^-}\mathbb{C}\widehat{\nabla}\bm{u}:\widehat{\nabla}\bm{
	\varphi}\, d\bm{x}- \langle
[\mathbb{C}\widehat{\nabla}\bm{u}\bm{n}],\bm{\varphi}\rangle_{(H^ { -\frac{1}{2} }
	(\Gamma\cap B_r(\bm{x}_0)),H^{\frac{1}{2}}(\Gamma\cap B_r(\bm{x}_0)))},
\end{equation*}
where $[,]$ denotes the jump across $\Gamma\cap
B_r(\bm{x}_0)$.
Then,using the fact that $[\mathbb{C}\widehat{\nabla}\bm{u}\bm{n}]=\bm{0}$ on $\Gamma\cap
B_r(\bm{x}_0)$ and $\widehat{\nabla}\bm{u}\in L^2(B_r(\bm{x}_0))$, we find that
\begin{equation*}
0=\int_{B_r(\bm{x}_0)}\mathbb{C}\widehat{\nabla}\bm{u}:\widehat{\nabla}\bm{
	\varphi}\, d\bm{x}. 
\end{equation*}
Therefore, $\bm{u}$ is a weak solution in $\Omega\setminus\overline{S}$, given that $\bm{x}_0$ and $\bm{\varphi}\in H^1_0(B_r(\bm{x}_0))$ are arbitrary.
\end{proof}

We note that, from the proof of the existence theorem above, a weak solution is
also a variational solution in the following sense:  $\bm{u}\in
H^1_{\Sigma}(\Omega^+)$, $\bm{u}\in {H}^1(\Omega^-)$,
$[\bm{u}]_{\Gamma}=\widetilde{\bm{g}}$ in $H^{\frac{1}{2}}(\Gamma)$ and
$\bm{u}$ is such that, for every $\bm{v}\in H^1_{\Sigma}(\Omega)$,
	\begin{equation}
		\int_{\Omega^+}\mathbb{C}\widehat{\nabla}\bm{u}\cdot
\widehat{\nabla}\bm{v}\, d\bm{x}+
\int_{\Omega^-}\mathbb{C}\widehat{\nabla}\bm{u}\cdot
\widehat{\nabla}\bm{v}\,d\bm{x}=0.
	\end{equation}
We observe that a variational solution could also be obtained by a
suitable lifting operator of the jump on $\Gamma$ to $\Omega\setminus \Gamma$, analogous to that utilized in the treatment of nonhomogeneous Dirichlet boundary conditions, reducing the problem to a source problem with homogeneous jump conditions on $S$ (see for example \cite{Lady,Zwieten_et_al13}).

\begin{corollary}
	There exists a unique solution $\bm{u}\in
H^1_{\Sigma}(\Omega\setminus\overline{S})$ to Problem \eqref{eq: Pu}.
\end{corollary}

We observe that other types of boundary conditions can, in principle, be imposed
on the buried part $\Sigma$ of $\partial\Omega$. For example, one can impose a
non-homogeneous traction there, modeling the load of contiguous rock formations
on $\Omega$ itself.

\begin{remark}
The approach to proving well-posedness for \eqref{eq: Pu} can be adapted to
other boundary value problems as well, such as Neumann problems
with non-homogeneous boundary conditions on $\partial\Omega$. In fact, given
$\bm{h}\in H^{-\frac{1}{2}}(\partial\Omega)$, one can show that there exists a
unique solution $\bm{u}_N\in \mathring{H}^1(\Omega\setminus\overline{S})$ for
the following problem:
		\begin{equation}\label{eq: Pu_nh}
			\begin{cases}
				\textrm{\textup{div}}\, (\mathbb{C}\widehat{\nabla}\bm{u}_N)=\bm{0},& \textrm{in}\,\,
				\Omega \setminus \overline{S},\\
				(\mathbb{C}\widehat{\nabla}\bm{u}_N)\bm{\nu}=\bm{h}, & \textrm{on}\,\, \partial \Omega, \\
				[\bm{u}_N]_{S}=\bm{g}, \\
				[(\mathbb{C}\widehat{\nabla}\bm{u}_N)\bm{n}]_{S}=\bm{0} ,\\
			\end{cases}
		\end{equation}
The proof of uniqueness in $\mathring{H}^1(\Omega\setminus\overline{S})$ follows exactly as in Theorem \ref{th: uniqueness}.
For the proof of existence, we notice that due to the linearity property of \eqref{eq: Pu_nh},
$\bm{u}_N$ can be decomposed as $\bm{u}_N:=\mathring{\bm{u}}+\bm{w}$, where
$\mathring{\bm{u}}\in\mathring{H}^1(\Omega\setminus\overline{S})$ is the unique solution to
\begin{equation}\label{eq: Pu_N}
\begin{cases}
\textrm{\textup{div}}\, (\mathbb{C}\widehat{\nabla}\mathring{\bm{u}})=\bm{0},& \textrm{in}\,\,
\Omega \setminus \overline{S},\\
(\mathbb{C}\widehat{\nabla}\mathring{\bm{u}})\bm{\nu}=\bm{0}, & \textrm{on}\,\, \partial \Omega, \\
[\mathring{\bm{u}}]_{S}=\bm{g}, \\
[(\mathbb{C}\widehat{\nabla}\mathring{\bm{u}})\bm{n}]_{S}=\bm{0} ,\\
\end{cases}
\end{equation}
and	$\bm{w}\in \mathring{H}^1(\Omega)$ is solution to
\begin{equation}\label{eq: Pw}
\begin{cases}
\textrm{\textup{div}}\, (\mathbb{C}\widehat{\nabla}\bm{w})=\bm{0},& \textrm{in}\,\,
\Omega \setminus \overline{S},\\
(\mathbb{C}\widehat{\nabla}\bm{w})\bm{\nu}=\bm{h}, & \textrm{on}\,\, \partial \Omega, \\
[\bm{w}]_{S}=\bm{0}, \\
[(\mathbb{C}\widehat{\nabla}\bm{w})\bm{n}]_{S}=\bm{0}.
\end{cases}
\end{equation}
The proof of existence of a solution
$\mathring{\bm{u}}\in\mathring{H}^1(\Omega\setminus\overline{S})$ for \eqref{eq:
Pu_N} then follows the same ideas as in  Theorem \ref{th: existence}, but with
the simplification that both $\bm{u}^+$ and $\bm{u}^-$ belong now to the same
space $\mathring{H}^1$. Problem \eqref{eq: Pw} is reduced to a
standard transmission problem, hence the existence of a unique solution in
$\mathring{H}^1(\Omega)$ follows easily.
\end{remark}

\section{The Inverse Problem: a uniqueness result}  \label{sec:inverse}
In this section we address the uniqueness for the inverse dislocation problem,
which consists in identifying the dislocation $S$ and the slip $\bm{g}$ on it
from displacement measurements made at the surface of the Earth. Uniqueness will
be proved under additional assumptions on the geometry and the data for Problem
\eqref{eq: Pu}. In particular, we consider a domain $\Omega$
which is partitioned in finitely-many Lipschitz subdomains, we assume that
the elasticity tensor is isotropic with
Lam\'e coefficients Lipschitz continuous in each subdomain, and we take the
dislocation surface to be a
graph with respect to a fixed, but arbitrary, coordinate frame. Such
assumptions are not unrealistic in the context of geophysical applications and
underscores the ill-posedness of the inverse problems without additional {\em a
priori} information.

Specifically, in additions to  \ref{hyp: assum_elastic_tens} and
\ref{hyp: assum_disl}, we assume the following:
\begin{description}
\item[\namedlabel{hyp: assum_domain}{Assumption 3} - domain and
partition] We denote by $\Xi\subseteq \partial\Omega\setminus
\overline{\Sigma}$ an open patch of the boundary where the measurements of
the displacement field are given.
Moreover, we assume that
\begin{equation*}
	\overline{\Omega}=\bigcup\limits_{k=1}^{N} \overline{D_k}
\end{equation*}
where $D_k$, for $k=1,\cdots N$, are pairwise non-overlapping bounded Lipschitz
domains.
We assume, without loss of generality, that $\Xi$ is contained in $\partial D_1$.
\end{description}

\begin{description}
\item[\namedlabel{hyp: assum_elastic_tens_ip}{Assumption 4} - elasticity
tensor] The elasticity tensor $\mathbb{C}=\mathbb{C}(\bm{x})$ is assumed
isotropic in each element of the partition of $\Omega$, i.e.,
\begin{equation}\label{ass: elast tensor}
	\mathbb{C}(\bm{x})=\sum_{k=1}^{N}\mathbb{C}_k(\bm{x}) \chi_{D_k}(\bm{x}),\qquad \mathbb{C}_k(\bm{x}):=\lambda_k(\bm{x})\mathbf{I}\otimes \mathbf{I}+2\mu_k(\bm{x})\mathbb{I},
\end{equation}
where $\lambda_k=\lambda_k(\bm{x})$ and $\mu_k=\mu_k(\bm{x})$, for
$k=1,\cdots, N$, are the Lam\'e coefficients related to the subdomain $D_k$,
$\mathbf{I}$ and $\mathbb{I}$ the identity matrix and the identity fourth-order
tensor, respectively. Each Lam\'e
parameter, $\lambda_k, \mu_k$, for $k=1,\cdots,N$, belongs to
$C^{0,1}(\overline{D_k})$, that is, there exists $M>0$ such that
\begin{equation}\label{ass: c0,1 regularity}
   \begin{aligned}
\|\mu_k\|_{C^{0,1}(\overline{D_k})}+\|\lambda_k\|_{C^{0,1}(\overline{D_k})}
\leq M,
   \end{aligned}
\end{equation}
	with
$\|\cdot\|_{C^{0,1}(\overline{D_k})}=\|\cdot\|_{L^{\infty}(\overline{D_k})}
+\|\nabla\cdot\|_{L^{\infty}
		(\overline{D_k})}$.
Finally, there
exist two positive constants
$\alpha_0, \beta_0$ such that,
\begin{equation}\label{ass: strong conv}
 \mu_k(\bm{x})\geq \alpha_0>0,\quad
  3\lambda_k(\bm{x})+2\mu_k(\bm{x})\geq \beta_0>0,\qquad \forall
  \bm{x}\in \overline{D_k}, \quad k=1,\ldots, N.
\end{equation}
These conditions ensures the uniform strong convexity of $\Omega$.

\medskip

{\item[\namedlabel{hyp: assum_S_graph}{Assumption 5}- further assumptions on the fault $S$]
The surface $S$ is assumed to be a graph of a Lipschitz function with respect to a given coordinate frame.}
\end{description}

Our main result for the inverse problem is the following theorem.

\begin{theorem}\label{th: uniq.ip}
Under \ref{hyp: assum_domain} and \ref{hyp: assum_elastic_tens_ip}, let
$S_1, S_2$ be as in \ref{hyp: assum_disl} and such that ${S}_1, {S}_2$ {satisfy \ref{hyp: assum_S_graph} with respect to the same coordinate frame}. Let $\bm{g}_i\in
H^{\frac{1}{2}}_{00}(S_i)$, for $i=1,2$, with \textup{Supp}$\,
	\bm{g}_i=\overline{S}_i$, for $i=1,2$, and $\bm{u}_i$, for $i=1,2$, be the unique solution of \eqref{eq: Pu} in
	$H^{1}_{\Sigma}(\Omega\setminus \overline{S})$ corresponding to $\bm{g}=\bm{g}_i$ and $S=S_i$.
	If $\bm{u}_{1\big|_{\Xi}}=\bm{u}_{2\big|
		_{\Xi}}$,  then $S_1=S_2$ and $\bm{g}_1=\bm{g}_2$.
\end{theorem}

\begin{remark} \label{r:faultgeometry} {The assumptions that the surfaces $S_i$, $i=1,2$, are graphs and that they are graphs with respect to the same coordinate frame is not overly restrictive in the context of faults in geophysics, which is the context of the present work. In fact, in a given geographical region faults tend to be approximately horizontal with respect to the surface of the Earth, as predominantly in {\em dip-slip faults}, or approximately vertical, as predominantly in {\em strike-slip faults}, although oblique faults can also occur, depending of the characteristics of the rock formations present (see e.g. \cite{FaultDistribution1,FaultDistribution2}). Furthermore, these assumptions exclude {\em a priori} the existence of internal faces common to both $S_1$ and $S_2$, when the faults enclose a bounded region of space (see Case \ref{i.unique2} below in the proof of the theorem). In the presence of such common faces, it seems difficult to prove uniqueness and it is not at all clear, in fact, that uniqueness does hold in this case. However, it is possible to prove uniqueness under other geometric conditions on the faults that also exclude common internal faces, for example if the fault surfaces are each a union of, at most, two rectangular faces. In the geophysical literature, often the fault is taken to be a {\em single} rectangular face.}

\end{remark}

We denote by $G$ the connected component of $\Omega\setminus \overline{S_1 \cup
S_2}$ containing $\Xi$. By definition we have that $G\subseteq \Omega\setminus
\overline{S_1 \cup S_2}$. In addition, we define
\begin{equation}\label{eq: mathcal_G}
\mathcal{G}:=\partial G \setminus \partial\Omega.
\end{equation}
Before proving Theorem \ref{th: uniq.ip}, we recall the following lemma proved
in \cite{Aspri-Beretta-Mazzucato-de-Hoop} in the special case where $\Omega$ is
a half-space. However, this result is clearly true for bounded domains as well.

\begin{lemma}\label{lem: surf}
	Let $S_1, S_2$ as in \ref{hyp: assum_disl} and such that ${S}_1, {S}_2$ {satisfy \ref{hyp: assum_S_graph} with respect to the same coordinate frame}. Then $\mathcal{G}=\overline{S_1\cup S_2}$.
\end{lemma}

\begin{proof}[Proof of Theorem \ref{th: uniq.ip}]
We proceed by contradiction and assume that $S_1\ne S_2$.
We first show that $\bm{w}:=\bm{u}_1-\bm{u}_2$ is identically zero in $G$. We
can assume, without loss of generality, that $\Xi$ is the graph of a Lipschitz
function in some coordinate frame, say with respect to the $z$-axis. In fact,
it is enough to take a possibly small open subset of $\Xi$ instead of the
entire $\Xi$, and then this hypothesis is always satisfied as $\partial\Omega$
is assumed globally Lipschitz.
On $\Xi$ we have that
	\begin{equation*}
		\bm{w}=\bm{0},\qquad
(\mathbb{C}\widehat{\nabla}\bm{w})\bm{\nu}=\bm{0},
	\end{equation*}
Then, fixing a point $\bm{x}_0\in\Xi$, we consider $B_R(\bm{x}_0)$, the ball of
radius $R$ and center $\bm{x}_0$, where $R$ is taken sufficiently small so
that $B_R(\bm{x}_0)\cap \Xi \subseteq \Xi$  and we denote by
$B^-_R(\bm{x}_0):=B_R(\bm{x}_0)\cap \overline{\Omega}$ and
$B^+_R(\bm{x}_0)=(B^-_R(\bm{x}_0))^C$, the complementary domain. We define
	\begin{equation}
		\widetilde{\bm{w}}:=
		\begin{cases}
		\bm{w} & \textrm{in}\,\, B^-_R(\bm{x}_0)\\
		\bm{0} & \textrm{in}\,\, B^+_R(\bm{x}_0).
		\end{cases}
	\end{equation}
We note that $\widetilde{\bm{w}}\in H^1(B_R(\bm{x}_0))$.\\
We observe next that, since $\Xi$ is the graph of a Lipschitz function, the
restriction of $\mathbb{C}$ on $\Xi$ is  Lipschitz as well. Then we can
extend $\mathbb{C}$ to a Lipschitz elasticity tensor $\widetilde{\mathbb{C}}$
in $B^-_R(\bm{x}_0)\cup B^+_R(\bm{x}_0)$ as follows: for each $\bm{\xi}$ on the
graph of $\Xi$, we extend $\mathbb{C}$ in $B^+_R(\bm{x}_0)$, keeping the
constant value $\mathbb{C}(\bm{\xi})$ along the vertical direction of the
coordinate frame. Note that this argument can be applied for each component of
the tensor. Consequently, arguing as in
\cite{Alessandrini-Rondi-Rosset-Vessella}, we obtain that $\widetilde{\bm{w}}$
is a weak solution of
	\begin{equation*}
		\textrm{div}(\widetilde{\mathbb{C}}\widehat{\nabla}\widetilde{\bm{w}})=\bm{0},\qquad \textrm{in}\,\,B_R(\bm{x}_0).
	\end{equation*}
We apply now the weak continuation property, see \cite{LinNakamuraWang}. In
fact, since $\widetilde{\bm{w}}=\bm{0}$ in $B^+_R(\bm{x}_0)$ and since the weak
continuation property holds in $B_R(\bm{x}_0)$, it follows  that
	\begin{equation*}
		\widetilde{\bm{w}}=\bm{0},\qquad \textrm{in}\,\, B_R(\bm{x}_0).
	\end{equation*}
In particular, $\bm{w}=\bm{0}$ in $B^-_R(\bm{x}_0)$. Furthermore,
applying again the weak continuation property, we find that $\bm{w}=\bm{0}$ in
$D_1$.

Next, thanks to the hypotheses on $S_i$, $i=1,2$, there exists a path-connected
open subdomain of $\Omega$ that connects $\Xi$ with every elements of the
partition which belong to $G$. Along this path, we can always
assume that the boundary of the partition is  Lipschitz. Consequently,
we can recursively apply the previous argument and  we get that $\bm{w}\equiv
\bm{0}$ in $G$. We then distinguish two cases:
	\begin{enumerate}[label=(\roman*), ref=(\roman*)]
	\item {$G=\Omega\setminus\overline{S_1\cup
S_2}$;}\label{i.unique1}
	\item {$G\subset \Omega\setminus\overline{S_1\cup S_2}$.}
\label{i.unique2}
\end{enumerate}
For  Case \ref{i.unique1}, see Figure \ref{fig:1},
\begin{figure}[h!]
	\centering
	\includegraphics[scale=0.6]{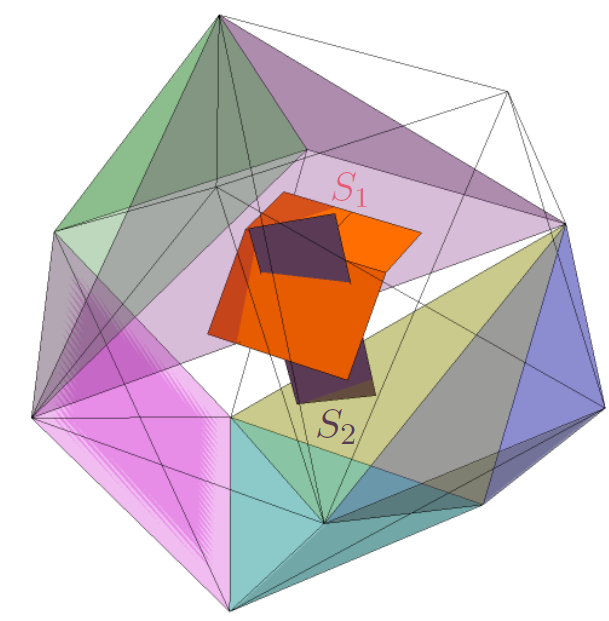}
	\caption{An example of the geometrical setting in Case \ref{i.unique1}. Graphics generated using Wolfram Mathematica$^{\text{\copyright}}$.}\label{fig:1}
\end{figure}
by the hypothesis that the surfaces are Lipschitz and the fact that $S_1\neq S_2$, without loss of generality, there exist a point $\bm{y}\in
{S}_1$ such that $\bm{y}\notin\overline{S}_2$, and a ball $B_r(\bm{y})$  that does not
intersect ${S}_2$, where $r$ is sufficiently small. Hence,
\begin{equation*}
\bm{0}=[\bm{w}]_{B_r(\bm{y})\cap S_1}=[\bm{u}_1]_{B_r(\bm{y})\cap S_1}=\bm{g}_1,
\end{equation*}
and this identity leads to a contradiction, as
\textrm{supp}($\bm{g}_1$)=$\overline{S}_1$. It follows that $\overline{S}_1=\overline{S}_2$ and, consequently,
\begin{equation*}
\bm{0}=[\bm{w}]_{S_1}=[\bm{w}]_{S_2}\Rightarrow [\bm{u}_1]_{S_1}=[\bm{u}_2]_{S_2}\Rightarrow
\bm{g}_1=\bm{g}_2.
\end{equation*}
Next, we analyze Case \ref{i.unique2}, see Figure \ref{fig:2}.
\begin{figure}[h!]
	\centering
	\includegraphics[scale=0.7]{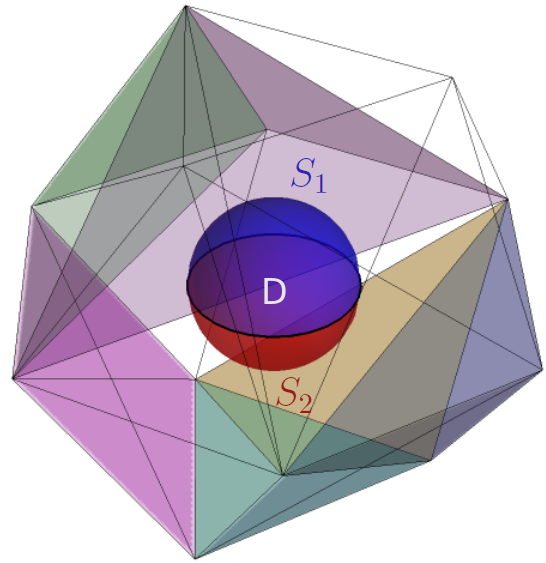}
	\caption{An example of the geometrical setting in Case \ref{i.unique2}. The bounded connected domain $D$ is such that $\partial D=\overline{S_1 \cup S_2}$. Graphics generated using Wolfram Mathematica$^{\text{\copyright}}$.}\label{fig:2}
\end{figure}
We recall that, by hypothesis, the two surfaces are Lipschitz graphs with respect to an arbitrary, but fixed, common frame. Then  by Lemma \ref{lem: surf} we can assume, without loss of generality, that
the complement of $\overline{S_1\cup S_2}$ has only one bounded connected component, since, if there are more than one, we can treat each one separately. That is, we can assume that there exists a  bounded connected domain $D$ such that $\partial D=\overline{S_1\cup S_2}$. In this situation, in particular, $\overline{S_1}$ and $\overline{S_2}$ intersects precisely only along their common boundary $\mathcal{C}=\partial S_1\cap \partial
S_2$, which is non empty. If there are other parts of their boundary that are not in common, they can be treated as in Case \ref{i.unique1}. Then
$\bm{w}=\bm{0}$ in a neighborhood of $\partial D$ in $\Omega\setminus
\overline{D}$, since $\bm{w}=\bm{0}$ in $G$.
The continuity of the tractions $(\mathbb{C}\widehat{\nabla}\bm{u}_1)\bm{n}$ and
$(\mathbb{C}\widehat{\nabla}\bm{u}_2)\bm{n}$ in trace sense across $S_1$ and
$S_2$, respectively, implies that
\begin{equation}\label{eq: conor_w-}
(\mathbb{C}\widehat{\nabla}\bm{w}^-)\bm{n}=\bm{0},
\end{equation}
in $H^{-\frac{1}{2}}(\partial D)$ and hence a.e. on $\partial D$,
where $\bm{w}^-$ indicates the function $\bm{w}$ restricted to $D$ and $\bm{n}$ the outward unit normal to $D$. Moreover, $\bm{w}^-$ satisfies
\begin{equation}\label{eq: equat_w-}
\textrm{div}(\mathbb{C}\widehat{\nabla}\bm{w}^-)=\bm{0}\,\quad \textrm{in}\,\, D.
\end{equation}
We conclude  from \eqref{eq: conor_w-} and \eqref{eq: equat_w-} that
$\bm{w}^-$ is in the kernel of the operator for elastostatics in $H^1(D)$,
i.e., it is a rigid motion:
\begin{equation*}
\bm{w}^-=\mathbf{A}\bm{x}+\bm{c},
\end{equation*}
where $\bm{c}\in\mathbb{R}^3$ and $\mathbf{A}\in\mathbb{R}^{3\times3}$ is a skew
matrix. We conclude the proof by showing that this rigid motion can only
be the trivial one.  By construction $\bm{w}^{-}=[\bm{w}]_{S_i}=\bm{g}_i$ on $S_i$, so in
particular it must vanish along $\partial S_i$, i.e., on $\mathcal{C}$ due to
the hypothesis $\bm{g}_i\in H^{\frac{1}{2}}_{00}(S_i)$. On the other hand the set of
solutions of the linear system $\mathbf{A}\,\bm{x}=\bm{c}$, for any given
$\bm{c}\in \RR^3$ is a one-dimensional linear subspace of $\RR^3$, since
$\mathbf{A}$ is anti-symmetric, and therefore it cannot contain a closed curve.
It follows that necessarily $\mathbf{A}=\mathbf{0}$ and $\bm{c}=\bm{0}$.
Consequently, $\bm{w}^-=\bm{0}$ in $D$, hence $[\bm{w}]=\bm{0}$ on
$\partial D$. In particular, $[\bm{w}]_{S_1}=\bm{0}=[\bm{u}_1]=\bm{g_1}\ne
\bm{0}$, by the assumption that $\text{Supp}(\bm{g}_i)=\overline{S_i}$.
We reach a contradiction and, therefore, Case \ref{i.unique2} does not occur.
\end{proof}

\bibliography{mybib}{}
\bibliographystyle{plain}

\end{document}